\documentclass[12pt,a4paper,twoside]{article}
\usepackage{graphics}
\usepackage{amssymb}
\usepackage{amsmath}
\usepackage{mathrsfs}
\usepackage{makeidx}
\usepackage{color}
\usepackage{graphicx}
\usepackage{subfigure}
\usepackage{multicol}
\usepackage{multirow}
\usepackage{float}
\usepackage{booktabs}
\usepackage{epsfig}
\usepackage{multirow}
\usepackage{epstopdf}
\usepackage{indentfirst}
\usepackage[colorlinks,
linkcolor=red,
anchorcolor=blue,
citecolor=blue
]{hyperref} 
\usepackage{caption}
\usepackage{algorithm,algorithmic}

\usepackage{latexsym, cite, bm, amsthm, amsmath}
\usepackage{enumerate}
\usepackage{marginnote}
\usepackage{epstopdf}

\def \[{\begin{equation}}
	\def \]{\end{equation}}

\newtheorem{thm}{Theorem}[section]
\newtheorem{cor}{Corollary}[section]
\newtheorem{lem}{Lemma}[section]

\newtheorem{alg}{Algorithm}[section]
\newtheorem{rem}{Remark}[section]

\newtheorem{exam}{Example}[section]
\numberwithin{equation}{section}

\title{\bf A generalization of the Newton-based matrix splitting iteration method for  generalized absolute value equations}

\usepackage[marginal]{footmisc}
\usepackage{authblk}
\author[a]{Xuehua Li\thanks{Email address: 3222714384@qq.com.}}
\author[a]{Cairong Chen\thanks{Corresponding author. Supported by the Natural Science Foundation of Fujian Province (Grand No. 2021J01661). Email address: cairongchen@fjnu.edu.cn.}}
\affil[a]{School of Mathematics and Statistics, FJKLMAA and Center for Applied Mathematics of Fujian Province, Fujian Normal University, Fuzhou, 350117, P.R. China}

\begin{document}
\date{\today}
\maketitle
\begin{quote}
{\bf Abstract:} A generalization of the Newton-based matrix splitting iteration method (GNMS) for solving the generalized absolute value equations (GAVEs) is proposed. Under mild conditions, the GNMS method converges to the unique solution of the GAVEs. Moreover, we can obtain a few weaker convergence conditions for some existing methods. Numerical results verify the effectiveness of the proposed method.
\end{quote}

\section{Introduction}\label{sec:intro}
Consider the system of generalized absolute value equations (GAVEs)
\begin{equation}\label{eq:gave}
Ax - B|x| - c = 0,
\end{equation}
where $A, B\in \mathbb{R}^{n\times n}$ and $c\in \mathbb{R}^n$ are given, and $x\in \mathbb{R}^n$ is unknown with $|x| = (|x_1|,~|x_2|,~\cdots,~|x_n|)^\top$. If the matrix $B$ is the identity matrix, GAVEs \eqref{eq:gave} turns into the system of absolute value equations (AVEs)
\begin{equation}\label{eq:ave}
	Ax - |x| - c = 0.
\end{equation}
To our knowledge, GAVEs~\eqref{eq:gave} was formally introduced by Rohn in $2004$ \cite{rohn2004}. Over the past two decades, GAVEs~\eqref{eq:gave} and AVEs~\eqref{eq:ave} have received considerable attention in the optimization community. The main reason is that GAVEs~\eqref{eq:gave} and AVEs~\eqref{eq:ave} are equivalent to the linear complementarity problem \cite{mang2007,mame2006,huhu2010,prok2009}, which has wide applications in engineering, science and economics \cite{copa1992}. As shown in \cite{mang2007}, solving GAVEs~\eqref{eq:gave} is NP-hard. In addition, if GAVEs~\eqref{eq:gave} is solvable, checking whether it has a unique solution or multiple solutions is NP-complete \cite{prok2009}. Nevertheless, some researches focused on constructing conditions under which the GAVEs~\eqref{eq:gave} has a unique solution for any $b\in \mathbb{R}^n$ \cite{rohf2014,mezz2020,wuli2020,wush2021,rohn2009,love2013}. Furthermore, the bounds for the solutions of GAVEs~\eqref{eq:gave} were studied in \cite{hlad2018}.

When GAVEs~\eqref{eq:gave} is solvable, considerable research effort has been, and is still, put into finding efficient algorithms for computing the approximate solution of it. For instance, by separating the differential and non-differential parts of GAVEs~\eqref{eq:gave}, Wang, Cao and Chen \cite{wacc2019} proposed the modified Newton-type (MN) iteration method, which is described in Algorithm~\ref{alg:mn}. Particularly, if $\Omega$ is the zero matrix, the MN iteration~\eqref{eq:mn} reduces to the Picard iteration method contained in \cite{rohf2014}. Whereafter, by using the matrix splitting technique, Zhou, Wu and Li \cite{zhwl2021} established the Newton-based matrix splitting (NMS) iteration method (see Algorithm~\ref{alg:nms}), which covers the MN method (by setting $\bar{M} = A$ and $\bar{N} = 0$), the shift splitting MN (SSMN) iteration method (see Algorithm~\ref{alg:ssmn}) \cite{liyi2021} (by setting $\bar{M}=\frac{1}{2}(A + \tilde{\Omega}),\bar{N}=\frac{1}{2}( \tilde{\Omega}-A)$ and $\Omega=0$) and the relaxed MN iteration method (see Algorithm~\ref{alg:rmn}) \cite{shzh2023} (by setting $\bar{M} = \theta A$ and $\bar{N}=(\theta - 1) A$).  More recently, Zhao and Shao proposed the relaxed NMS (RNMS) iteration method \cite{zhsh2023}, which is shown in Algorithm~\ref{alg:rnms}. On the one hand, as mentioned in \cite{zhsh2023}, if $\theta = 1$, $\hat{M} = \bar{M}$, $\hat{N} = \bar{N}$ and $\hat{\Omega} = \Omega$, the RNMS method is simplified to the NMS method. On the other hand, if $\bar{M} = \hat{\Omega} + \theta \hat{M}$, $\bar{N} = \hat{\Omega} + (\theta -1) \hat{M} + \hat{N}$ and $\Omega =0$, the NMS method is reduced to the RNMS method. In this sense, we can
conclude that the NMS method is equivalent to the RNMS method. For more numerical algorithms, one can refer to \cite{lild2022,soso2023,jizh2013,tazh2019,
cyhm2023,acha2018,alct2023,lilw2018} and the references therein.

\begin{alg}\label{alg}\label{alg:mn}
Assume that $\Omega\in \mathbb{R}^{n\times n}$ is a given matrix such that $A + \Omega$ is nonsingular.  Given an initial vector $x^0\in \mathbb{R}^{n}$, for $k=0,1,2,\cdots$ until the iteration sequence $\{x^k\}^\infty_{k=0} $ is convergent, compute
\begin{equation}\label{eq:mn}
		x^{k+1}=(A + \Omega)^{-1}(\Omega x^k+B|x^k| + c).
\end{equation}
\end{alg}

\begin{alg}\label{alg}\label{alg:nms}
Assume that $x^0\in \mathbb{R}^{n}$ is an arbitrary initial vector. Let $A = \bar{M}-\bar{N}$ and $\Omega\in \mathbb{R}^{n\times n}$ be a given matrix such that $\bar{M} + \Omega$ is nonsingular. For $k=0,1,2,\cdots$ until the iteration sequence $\{x^k\}^\infty_{k=0} $ is convergent, compute
\begin{equation}\label{eq:nms}
		x^{k+1}=(\bar{M} + \Omega)^{-1}\left[(\bar{N}+\Omega) x^k+B|x^k| + c\right].
\end{equation}
\end{alg}

\begin{alg}\label{alg}\label{alg:ssmn}
Assume that $x^0\in \mathbb{R}^{n}$ is an arbitrary initial vector. Let $\tilde{\Omega}\in \mathbb{R}^{n\times n}$ be a given matrix such that $A + \tilde{\Omega}$ is nonsingular. For $k=0,1,2,\cdots$ until the iteration sequence $\{x^k\}^\infty_{k=0} $ is convergent, compute
\begin{equation}\label{eq:ssmn}
		x^{k+1}=(A + \tilde{\Omega})^{-1}\left[(\tilde{\Omega}-A) x^k+2B|x^k| + 2c\right].
\end{equation}
\end{alg}

\begin{alg}\label{alg}\label{alg:rmn}
Assume that $x^0\in \mathbb{R}^{n}$ is an arbitrary initial vector. Choose $\Omega\in \mathbb{R}^{n\times n}$ and $\theta\ge 0$ such that $\theta A + \Omega$ is nonsingular. For $k=0,1,2,\cdots$ until the iteration sequence $\{x^k\}^\infty_{k=0} $ is convergent, compute
\begin{equation}\label{eq:rmn}
		x^{k+1}=(\theta A +  \Omega)^{-1}\left[\Omega x^k+(\theta -1) A x^k + B|x^k| + c\right].
\end{equation}
\end{alg}

\begin{alg}\label{alg}\label{alg:rnms}
Assume that $x^0\in \mathbb{R}^{n}$ is an arbitrary initial vector. Let $A = \hat{M}-\hat{N}$ and $\hat{\Omega}\in \mathbb{R}^{n\times n}$ be a given matrix such that $\theta \hat{M} + \hat{\Omega}$ $(\theta \ge 0)$ is nonsingular. For $k=0,1,2,\cdots$ until the iteration sequence $\{x^k\}^\infty_{k=0} $ is convergent, compute
\begin{equation}\label{eq:rnms}
		x^{k+1}=(\theta \hat{M} + \hat{\Omega})^{-1}\left[(\hat{\Omega} + (\theta -1) \hat{M} + \hat{N}) x^k+B|x^k| + c\right].
\end{equation}
\end{alg}

This paper is devoted to developing a generalization of the NMS (GNMS) iteration method. To this end, the matrix splitting technique, the relaxation technique and the variable  transformation $Qy = |x|$ (where $Q$ is a nonsingular matrix) are exploited. The convergence of the proposed method is studied and numerical examples are given to demonstrate the efficiency of the GNMS method.

The paper is organized as follows. In Section \ref{sec:prel}, we present notations and some useful lemmas. In Section \ref{sec:main}, we propose the GNMS iterative method for solving the GAVEs~\eqref{eq:gave} and a number of its special cases. In addition the convergence analysis of the proposed method will be discussed in detail. In Section \ref{sec:Numericalexample}, we give numerical results to show the effectiveness of our method. Finally, we give a conclusion for this paper in Section \ref{sec:Conclusions}.

\section{Preliminaries}\label{sec:prel}
We recall some basic definitions and results that will be used later in this paper.

\textbf{Notation.}
 In this paper, let $\mathbb{R}^{n\times n}$ be the set of all $n \times n$ real matrices and $\mathbb{R}^{n}= \mathbb{R}^{n\times 1}$. $|U|\in\mathbb{R}^{m\times n}$ denote the componentwise absolute value of matrix $U$. $I$ denotes the identity matrix with suitable  dimensions. $\Vert U\Vert$ denotes the $2$-norm of $U\in \mathbb{R}^{m\times n}$ which is defined by the formula $\Vert U\Vert= \max\{\Vert Ux\Vert: x\in\mathbb{R}^n,~\Vert x\Vert=1\}$, where $\Vert x\Vert$ is the $2$-norm of the vector $x$. For any matrix $U=(u_{ij})$ and $V=(v_{ij})\in\mathbb{R}^{n\times n}$, $U\leq V$ means that $u_{ij}\leq v_{ij}$ for any $i,~j=1,2,\cdots,n$. $\rho(U)$ denotes the spectral radius of $U$. A matrix $U$ is positive semi-definite if $\langle x,Ux\rangle\geq0$ for any nonzero vector $x\in\mathbb{R}^n$. For any matrix $U\in\mathbb{R}^{n\times n}$, $U = U_1-U_2$ is call a splitting of $U$ if $U_1$ is nonsingular.

\begin{lem}\label{lemma:1}
	\cite[Lemma 2.1]{young1971} If $s$ and $q$ are real, then both roots of the quadratic equation
	\begin{equation*}
		x^2 - sx + q = 0
	\end{equation*}
	are less than one in modulus if and only if
	\begin{equation*}
		|q| < 1\quad \text{and} \quad |s| < 1 + q.
	\end{equation*}
\end{lem}

\begin{lem}\label{lemma:2}
	\cite{harj1985} For $x,~y \in \mathbb{R}^n$, the following results hold:
	\begin{itemize}
		\item[(a)] $\Vert\vert x\vert-\vert y\vert\Vert \leq \Vert x - y \Vert$;
		\item[(b)] if $0 \leq x \leq y$, then $\Vert x\Vert\leq \Vert y\Vert$;
		\item[(c)] if $x \leq y$ and $U \geq 0$, then $Ux \leq Uy$.
	\end{itemize}
\end{lem}


\begin{lem}\label{lemma:4}\cite[Theorem 1.5]{ys2003}
For $U\in \mathbb{R}^{n\times n}$, the series $\sum\limits_{k=0}^{\infty}U^k$ converges if only if $\rho(U)<1$ and we have $\sum\limits_{k=0}^{\infty}U^k=(I-U)^{-1}$ whenever it converges.
\end{lem}

\begin{lem}\label{lem:conv}
	\cite[Corollary 4.4.1]{serre2002} For $U\in \mathbb{R}^{n\times n}$, $\lim\limits_{k \to +\infty}U^k=0$ if and only if $\rho(U)<1.$
\end{lem}


\begin{lem}\label{lemma:5}\cite[Exercise 8.1.16]{harj1985} For any matrices $U,~V \in\mathbb{R}^{n\times n} $, if $0 \leq U \leq V$, then $\Vert U\Vert \leq \Vert V\Vert$.
\end{lem}

\begin{lem}\label{lemma:6}\cite[Proposition 4.1.6]{serre2002}
For any $U\in\mathbb{R}^{n\times n}$ and the induced norm $\Vert \cdot\Vert$ on $\mathbb{R}^n$, we have  $\rho(U)\leq\Vert U\Vert.$
\end{lem}

\section{The GNMS method and its convergence analysis}\label{sec:main}
Let $|x|=Qy$\footnote{This variable transformation was formally introduced in the modified fixed point iteration method for solving AVEs~\eqref{eq:ave} \cite{yuch2022} and further used in \cite{zhzl2023}.}  with $Q\in \mathbb{R}^{n\times n}$ being invertible. Then the GAVEs~\eqref{eq:gave} can be rewritten as
\begin{equation}\label{eq:two}
\begin{cases}
	Qy - |x| = 0  ,\\
	Ax - BQy = c.
\end{cases}
\end{equation}
By splitting matrices $A$ and $Q$ as
$$ A= M-N \quad \text{and}\quad Q = Q_1-Q_2$$
with $M$ and $Q_1$ being nonsingular, the two-by-two nonlinear system \eqref{eq:two} is equivalent to
\begin{equation}\label{eq:etwo}
	\begin{cases}
		Q_1 y= Q_1 y -\tau Q_1y + \tau Q_2y + \tau |x|,\\
		Mx = Nx + BQ_1 y- BQ_2 y + c,
	\end{cases}
\end{equation}
in which $\tau$ is a positive constant. According to \eqref{eq:etwo}, we can develop the following iteration method for solving the GAVEs~\eqref{eq:gave}.

\begin{alg}\label{alg}\label{alg:gnms}
Let $Q\in \mathbb{R}^{n\times n}$ be a nonsingular matrix and $A=M-N$, $Q = Q_1-Q_2$ with $M$ and $Q_1$ being invertible. Given initial vectors $x^0\in \mathbb{R}^{n}$ and $y^0\in \mathbb{R}^{n}$, for $k=0,1,2,\cdots$ until the iteration sequence $\{(x^k,y^k)\}^\infty_{k=0} $ is convergent, compute
	\begin{equation}\label{eq:gnms}
		\begin{cases}
		y^{k+1}=(1-\tau)y^k+\tau{Q_1}^{-1}\left(Q_2y^k+ |x^k|\right),\\
			x^{k+1}=M^{-1}(Nx^k+BQ_1y^{k+1}-BQ_2y^k + c),
		\end{cases}
	\end{equation}
where the relaxation parameter $\tau>0$.
\end{alg}

Unlike the methods proposed in \cite{wacc2019,zhwl2021,shzh2023,liyi2021,zhsh2023}, the iteration scheme \eqref{eq:gnms} updates $y$ first and then $x$. On the one hand, surprisingly, the iteration scheme \eqref{eq:gnms} reduces to the NMS iteration \eqref{eq:nms} whenever $Q=Q_1 = I, Q_2=0, M = \bar{M} + \Omega, N = \bar{N} + \Omega$ and $\tau = 1$. On the other hand, the iteration scheme \eqref{eq:gnms} frequently can not be rewritten as the form of the NMS iteration \eqref{eq:nms}. Indeed, the iteration scheme \eqref{eq:gnms} can be reformulated as
\begin{equation}\label{eq:rgnms}
\begin{cases}
	y^{k+1}=(1-\tau)y^k+\tau{Q_1}^{-1}Q_2y^k+\tau{Q_1}^{-1}|x^k|,\\
	x^{k+1}=M^{-1}(Nx^k+ (1-\tau)BQy^{k} + \tau B|x^k| + c),
\end{cases}
\end{equation}
which can not be reduced to \eqref{eq:nms} provided that $\tau \neq 1$. That is why the proposed method is called as the GNMS iteration method. According to the statements in Section~\ref{sec:intro}, by approximate choosing $M,N,Q_1,Q_2$ and $\tau$, it is easy to see that the GNMS iteration method also involves the Picard iteration method, the MN method, the SSMN method, the RNMS method and their relaxation versions (if exist) as special cases.

Now we are in the position to show the convergence of the GNMS iterative method. Denote
\begin{equation}\label{eq:nota}
\alpha=\Vert Q_1^{-1}Q_2\Vert,~\beta=\Vert Q_1^{-1}\Vert,~\gamma=\Vert M^{-1}N\Vert,~\mu=\Vert M^{-1}BQ_1\Vert,~\nu=\Vert M^{-1}BQ_2\Vert.
\end{equation}
Then we have the following convergence theorem.

\begin{thm}\label{theorem:5}
Let $A = M - N$ and $Q = Q_1 - Q_2$ with $M$ and $Q_1$ being  nonsingular matrices. If
\begin{equation}\label{eq:con1}
\big\vert \gamma\vert 1-\tau\vert+\tau(\gamma\alpha-\beta \nu)\big\vert < 1\quad \text{and}\quad \tau(\mu\beta+\beta \nu) < (\gamma-1)(\vert 1-\tau\vert+\tau\alpha-1),
\end{equation}
then the GAVEs~\eqref{eq:gave} has a unique solution $x^*$ and the sequence $\{(x^k,~y^k)\}^\infty_{k=0}$ generated by~\eqref{eq:gnms} converges to $(x^*,y^* = Q^{-1}|x^*|)$.
\end{thm}

\begin{proof}
It follows from \eqref{eq:gnms} that
\begin{equation}\label{eq:gnmsm}
		\begin{cases} y^{k}=(1-\tau)y^{k-1}+\tau{Q_1}^{-1}\left(Q_2y^{k-1}+ |x^{k-1}|\right),\\
			x^{k}=M^{-1}(Nx^{k-1} +BQ_1y^{k}-BQ_2y^{k-1} + c),
		\end{cases}
	\end{equation}
Subtracting \eqref{eq:gnmsm} from \eqref{eq:gnms}, we have
\begin{equation}\label{eq:mm}
		\begin{cases} y^{k+1} - y^{k} = (1-\tau)(y^k - y^{k-1})+\tau{Q_1}^{-1}\left[Q_2(y^k - y^{k-1}) +(|x^k| - |x^{k-1}|)\right],\\
			x^{k+1} - x^{k} = M^{-1}\left[(N(x^k - x^{k-1}) +BQ_1(y^{k+1} - y^{k}) -BQ_2(y^k - y^{k-1})\right],
		\end{cases}
	\end{equation}
from which, Lemma~\ref{lemma:2}~(a) and \eqref{eq:nota}, we have
\begin{equation}\label{neq:ss}
		\begin{aligned}
			\Vert y^{k+1}-y^k\Vert&
			\leq\vert 1-\tau\vert\Vert y^k - y^{k-1} \Vert+\tau\alpha\Vert y^k - y^{k-1}\Vert+\tau \beta\Vert\Vert x^k - x^{k-1}\Vert,\\
			\Vert x^{k+1} - x^k\Vert&
			\leq\gamma \Vert x^k - x^{k-1}\Vert
+\mu \Vert y^{k+1} - y^k \Vert+\nu  \Vert y^k - y^{k-1}\Vert,
		\end{aligned}
	\end{equation}
from which we have
	\begin{equation}\label{eq:wss}
		\begin{bmatrix}
			1 & 0\\
			-\mu & 1
		\end{bmatrix}
		\begin{bmatrix}
			\Vert y^{k+1} - y^k\Vert\\
			\Vert x^{k+1} -x^k\Vert
		\end{bmatrix}
		\leq
		\begin{bmatrix}
			\vert1-\tau\vert+\tau\alpha & \tau\beta\\
			 \nu & \gamma
		\end{bmatrix}
		\begin{bmatrix}
			\Vert y^k - y^{k-1}\Vert\\
			\Vert x^k - x^{k-1}\Vert
		\end{bmatrix}.			
	\end{equation}
Multiplying both sides of \eqref{eq:wss} from the left by the nonnegative matrix
$
	P =
	\begin{bmatrix}
		1 & 0\\
		\mu & 1
	\end{bmatrix}
$ and using Lemma \ref{lemma:2}~(c), we have
 \begin{equation}\label{neq:ssnorm}
 	\begin{bmatrix}
 		\Vert y^{k+1} - y^k\Vert\\
 		\Vert x^{k+1} - x^k\Vert
 	\end{bmatrix}
 	\leq W
 	\begin{bmatrix}
 		\Vert y^k - y^{k-1}\Vert\\
 		\Vert x^k - x^{k-1}\Vert
 	\end{bmatrix},
 \end{equation}
 where
$$W=
	\begin{bmatrix}
		\vert1-\tau\vert+\tau\alpha & \tau\beta\\
		(\vert1-\tau\vert+\tau\alpha)\mu+\nu & \tau\beta \mu+\gamma
	\end{bmatrix}.
$$
For each $m\geq1$, if $\rho(W)<1$, it follows from \eqref{neq:ssnorm}, Lemma~\ref{lemma:4} and Lemma~\ref{lem:conv} that
	\begin{align}\label{neq:cauchy}
		\begin{bmatrix}
			\Vert y^{k+m}-y^k \Vert\\
			\Vert x^{k+m}-x^k \Vert
		\end{bmatrix}
		&=
		\begin{bmatrix}
			\Big\Vert \sum\limits_{j=0}^{m-1} (y^{k+j+1}-y^{k+j}) \Big\Vert\\
			\Big\Vert\sum\limits_{j=0}^{m-1} (x^{k+j+1}-x^{k+j}) \Big\Vert
		\end{bmatrix}
		\leq
		\begin{bmatrix}
			\sum\limits_{j=0}^{\infty}\Vert (y^{k+j+1}-y^{k+j}) \Vert\\
			\sum\limits_{j=0}^{\infty}\Vert (x^{k+j+1}-x^{k+j}) \Vert
		\end{bmatrix}\nonumber\\
		&\leq \sum_{j=0}^{\infty} W^{j+1}
		\begin{bmatrix}
			\Vert y^k-y^{k-1}\Vert\\
			\Vert x^k-x^{k-1}\Vert
		\end{bmatrix}
		=
		(I-W)^{-1}W
		\begin{bmatrix}
			\Vert y^k-y^{k-1}\Vert\\
			\Vert x^k-x^{k-1}\Vert
		\end{bmatrix}\nonumber\\
		&\leq(I-W)^{-1}W^k
		\begin{bmatrix}
			\Vert y^1-y^0\Vert\\
			\Vert x^1-x^0\Vert
		\end{bmatrix}\rightarrow \begin{bmatrix}
			0\\
			0
		\end{bmatrix} (\text{as} \quad k\rightarrow \infty).
	\end{align}
Thus, both $\{y^k\}^\infty_{k=0}$ and $\{x^k\}^\infty_{k=0}$ is Cauchy Sequence whenever $\rho(W)<1$. Then, from \cite[Theorem 5.4.10]{harj1985}, $\{x^k\}^\infty_{k=0}$ and $\{y^k\}^\infty_{k=0}$ are  convergent. Let $\lim_{k\rightarrow\infty}y^k = y^*$ and $\lim_{k\rightarrow\infty}x^k = x^*$. Then it follows from \eqref{eq:gnms} that
\begin{equation}\label{eq:lgnms}
		\begin{cases}
		y^*=(1-\tau)y^*+\tau{Q_1}^{-1}\left(Q_2y^*+ |x^*|\right),\\
			x^*=M^{-1}(Nx^*+BQ_1y^*-BQ_2y^* + c),
		\end{cases}
	\end{equation}
which implies that
\begin{equation}\label{eq:lgnms}
		\begin{cases}
		Qy^*= |x^*|,\\
		Ax^* - B|x^*| - c = 0,
		\end{cases}
	\end{equation}
that is, $x^*$ is a solution to GAVE~\eqref{eq:gave}. In the following, we will prove that $\rho(W)<1$ if \eqref{eq:con1} holds. For simplicity, we denote the matrix $W$ as
$$
	W=
	\begin{bmatrix}
		f & g\\
		f\mu+\nu & g\mu+\gamma\\
	\end{bmatrix},
$$
where
	\begin{equation}\label{eq:fg}
		f=\vert1-\tau\vert+\tau\alpha \quad \text{and} \quad g=\tau\beta.
	\end{equation}
Suppose that $\lambda$ is an eigenvalue of $W$, then
	\begin{equation}\label{eq:det}
		\det(\lambda I - W) = \det
		\begin{bmatrix}
		\lambda-f & -g\\
			-f\mu-\nu & \lambda-(g\mu+\gamma)
		\end{bmatrix}
	=0,
	\end{equation}
from which we have
\begin{equation*}
	\lambda^2-(f+\mu g+\gamma)\lambda+\gamma f-g\nu=0,
\end{equation*}
it follows from Lemma \ref{lemma:1} that $|\lambda|<1$ if and only if
$$
\vert \gamma f-g\nu\vert<1\quad \text{and} \quad
		f+\mu g+\gamma<1+\gamma f-g\nu,
$$
that is
$$
		\big\vert \gamma\vert 1-\tau\vert+\tau(\gamma\alpha-\beta \nu)\big\vert < 1\quad \text{and} \quad
		\tau(\mu\beta+\beta \nu) < (\gamma-1)(\vert 1-\tau\vert+\tau\alpha-1),
$$
which is \eqref{eq:con1}.

Finally, we will prove the unique solvability. In contrast, suppose that $x^*$ and $\bar{x}^*$ are two different solutions of the GAVEs~\eqref{eq:gave}, then we have
	\begin{subequations}
		\begin{align}
			\Vert y^*-\bar{y}^*\Vert&\leq\vert1-\tau\vert\Vert y^*-\bar{y}^*\Vert+\tau\alpha\Vert y^*-\bar{y}^*\Vert+\tau\beta\Vert x^*-\bar{x}^*\Vert,\label{3.9a}\\
			\Vert x^*-\bar{x}^*\Vert&\leq \gamma\Vert x^*-\bar{x}^*\Vert+\mu\Vert y^*-\bar{y}^*\Vert+\nu\Vert y^*-\bar{y}^*\Vert \label{3.9b},
		\end{align}   	
	\end{subequations}
where $y^* = Q^{-1}|x^*|$ and $\bar{y}^* = Q^{-1}|\bar{x}^*|$. Note that it can be deduced from \eqref{eq:con1} that $\gamma<1$ and $|1-\tau|+\tau\alpha<1$, it follows from \eqref{3.9a} and \eqref{3.9b} that
	\begin{subequations}
		\begin{align}
			\Vert y^*-\bar{y}^*\Vert&\leq\dfrac{\tau\beta}{1-(|1-\tau|+\tau\alpha)}\Vert x^*-\bar{x}^*\Vert,\label{3.10a}\\
			\Vert x^*-\bar{x}^*\Vert&\leq\dfrac{\mu+\nu}{1-\gamma}\Vert y^*-\bar{y}^*\Vert.\label{3.10b}
		\end{align}
	\end{subequations}
It follows from \eqref{3.9a}, \eqref{3.10b} and the second inequality of \eqref{eq:con1} that
\begin{eqnarray}\label{neq: y*}
		\Vert y^*-\bar{y}^*\Vert&\leq&\vert1-\tau\vert\Vert y^*-\bar{y}^*\Vert+\tau\alpha\Vert y^*-\bar{y}^*\Vert+\dfrac{\tau\beta(\mu+\nu)}{1-\gamma}\Vert y^*-\bar{y}^*\Vert\nonumber\\
		&<&\vert1-\tau\vert\Vert y^*-\bar{y}^*\Vert+\tau\alpha\Vert y^*-\bar{y}^*\Vert+[1-(|1-\tau|+\tau\alpha)]\Vert y^*-\bar{y}^*\Vert\nonumber\\
		&=&\Vert y^*-\bar{y}^*\Vert,
\end{eqnarray}
which will lead to a contradiction whenever $y^*\neq \bar{y}^*$ (since $x^*\neq \bar{x}^*$). Hence, we have $x^*=\bar{x}^*$.
\end{proof}

\begin{cor}\label{cor:conv}
If
\begin{equation}\label{eq:con2}
\vert \gamma\alpha-\beta \nu\vert<\gamma<1, \beta(\mu+\nu)<(\gamma-1)(\alpha-1),
0<\tau <\dfrac{2(1-\gamma)}{\beta(\mu+\nu)-(\gamma-1)(\alpha+1)},
\end{equation}
then the GAVEs~\eqref{eq:gave} has a unique solution $x^*$ and the sequence $\{(x^k,~y^k)\}^\infty_{k=0}$ generated by~\eqref{eq:gnms} converges to $(x^*,y^* = Q^{-1}|x^*|)$.
\end{cor}

\begin{proof}
In order to prove this corollary, it suffices to prove that \eqref{eq:con2} implies \eqref{eq:con1}. We will prove it in the following and the proof is divided into two cases.
\begin{itemize}
  \item [] \textbf{Case I}: We first consider $0 < \tau\leq 1$. It can be inferred from the first two inequalities of \eqref{eq:con2} that
	$$\dfrac{1-\gamma}{\gamma\alpha-\beta\nu-\gamma}<0\quad \text{and}\quad 1<\dfrac{-\gamma-1}{\gamma\alpha-\beta \nu-\gamma},$$
from which we have
\begin{equation}\label{neq:tau}
\dfrac{1-\gamma}{\gamma\alpha-\beta\nu-\gamma}<\tau<\dfrac{-\gamma-1}{\gamma\alpha-\beta \nu-\gamma}.
\end{equation}
Multiplying both sides of \eqref{neq:tau} by $\gamma\alpha-\beta\nu-\gamma$, we get
\begin{equation*}
		-\gamma-1< \tau(\gamma\alpha-\beta\nu)-\tau\gamma <1-\gamma.
\end{equation*}
For the right inequality, we have
$$\tau(\gamma\alpha-\beta \nu)+(1-\tau)\gamma<1.$$
For the left inequality, we have
$$ \tau(\gamma\alpha-\beta \nu)+(1-\tau)\gamma>-1.$$
In conclusion, we have
\begin{equation}\label{eq:ncond1}
\big\vert\tau(\gamma\alpha-\beta \nu)+\vert1-\tau\vert \gamma\big\vert<1.
\end{equation}
Furthermore,
\begin{align}\nonumber
		\tau(\mu\beta+\beta \nu) - (\gamma-1)(\vert 1-\tau\vert+\tau\alpha-1)&=\tau(\mu\beta+\beta \nu)- (\gamma-1)( 1-\tau+\tau\alpha-1)\\\nonumber
		&=\tau(\mu\beta+\beta \nu)) -\tau(\gamma-1)(\alpha-1)\\\label{eq:ncond2}
		&<0.
	\end{align}
Obviously, \eqref{eq:ncond1} and \eqref{eq:ncond2} imply \eqref{eq:con1}.

  \item [] \textbf{Case II}: We consider $1<\tau <\dfrac{2(1-\gamma)}{\beta(\mu+\nu)-(\gamma-1)(\alpha+1)}$, from which and the first two inequalities of \eqref{eq:con2} we have
      $$ \dfrac{\gamma-1}{\gamma\alpha-\beta \nu+\gamma}<\tau,$$
      which implies that
      \begin{equation}\label{neq:-1}
		\tau(\gamma\alpha-\beta \nu)+(\tau-1)\gamma>-1.
	\end{equation}
In addition, since $\beta(\mu+\nu)-(\gamma-1)(\alpha+1)>0$ and $\gamma\alpha-\beta \nu+\gamma >0$, it follows that $$\tau<\dfrac{2(1-\gamma)}{\beta(\mu+\nu)-(\gamma-1)(\alpha+1)}<\dfrac{\gamma+1}{\gamma\alpha-\beta \nu+\gamma},$$
which implies that
\begin{equation}\label{neq:1}
		\tau(\gamma\alpha-\beta \nu)+(\tau-1)\gamma<1.
	\end{equation}
Combining \eqref{neq:-1} and \eqref{neq:1}, we have
\begin{equation}\label{eq:ncond3}
\big\vert\tau(\gamma\alpha-\beta \nu)+\vert1-\tau\vert \gamma\big\vert<1.
\end{equation}
In addition, it follows from $\tau <\dfrac{2(1-\gamma)}{\beta(\mu+\nu)-(\gamma-1)(\alpha+1)}$ that
\begin{align}\nonumber
		\tau\beta(\mu+\nu)&<2(1-\gamma)+\tau(\gamma-1)(\alpha+1)\\\nonumber
		&=(\gamma-1)(\tau\alpha+\tau-2)\\\nonumber
		&=(\gamma-1)(\tau-1+\tau\alpha-1)\\\label{eq:ncond4}
		&=(\gamma-1)(\vert1-\tau\vert
		+\tau\alpha-1).
\end{align}
It easy to see that \eqref{eq:ncond3} and \eqref{eq:ncond4} imply \eqref{eq:con1}.
\end{itemize}

The proof is completed by summarizing the results of Case I and Case II.
\end{proof}

If $M = A + \Omega$, $N = \Omega$, where $\Omega$ is a semi-definite matrix, $Q = Q_1 = I$ and $\tau = 1$, then the GNMS method reduces to the MN method and we have the following Corollary.
\begin{cor}\label{cor:mn}
	Let $A+\Omega$ be nonsingular and
\begin{equation}\label{eq:mnconv}
\Vert(A+\Omega)^{-1}\Omega\Vert+\Vert(A+\Omega)^{-1}B\Vert < 1.
 \end{equation}
 Then, the GAVEs \eqref{eq:gave} has a unique solution $x^*$ and the MN iteration method \eqref{eq:mn} converges to the unique solution.
\end{cor}
\begin{proof}
In this case, the condition \eqref{eq:con2} reduces to $\Vert(A+\Omega)^{-1}\Omega\Vert+\Vert(A+\Omega)^{-1}B\Vert < 1$. Then the results follow from Corollary~\ref{cor:conv}.
\end{proof}

\begin{rem}
In \cite[Theorem 3.1]{wacc2019}, the authors show that the MN iteration method~\eqref{eq:mn} converges linearly to a solution of the GAVEs~\eqref{eq:gave} if
\begin{equation}\label{eq:wc}
\|(A+\Omega)^{-1}\|(\|\Omega\| + \|B\|)<1.
\end{equation}
However, under the condition \eqref{eq:wc}, the unique solvability of the GAVEs~\eqref{eq:gave} does not explored in \cite{wacc2019}. In addition, \eqref{eq:wc} implies \eqref{eq:mnconv} but the converse is generally not true. Hence, the condition \eqref{eq:mnconv} is weaker than \eqref{eq:wc} and we can conclude that the GAVEs~\eqref{eq:gave} is unique solvable under \eqref{eq:wc}.

\end{rem}

If $\Omega=0$, then the MN method develops into the Picard method \cite{rohf2014} and we get the following Corollary $\ref{cor:pi}$.
\begin{cor}\label{cor:pi}
	Let $A$ be nonsingular and $\Vert A^{-1}B\Vert< 1$. Then, the Picard iterative method converges to the unique solution $x^*$ of the GAVEs \eqref{eq:gave}.
\end{cor}

\begin{rem}
In \cite[Theorem~2]{rohf2014}, the authors showed that the Picard iteration method converges to the unique solution of the GAVEs~\eqref{eq:gave} if $\rho(|A^{-1}B|)<1$. Clearly, when $A^{-1}B\geq0$, $\rho(|A^{-1}B|)\leq\Vert A^{-1}B\Vert< 1$, but the converse is generally not true. However, when $A^{-1}B \ngeq 0$, the following example shows that $\rho(|A^{-1}B|)< 1$ and $\Vert A^{-1}B\Vert< 1$ are irrelevant.


\begin{exam}
		\begin{enumerate}
			\item[(1)] Consider $A=
			\begin{bmatrix}
				1 & 0.5\\
				3 & 0.25
			\end{bmatrix}$
			and $B=
			\begin{bmatrix}
				1 & 0\\
				2.1 & 1
			\end{bmatrix}$, it follows that $A^{-1}B=
			\begin{bmatrix}
				0.64 & 0.4\\
				0.72 & -0.8
			\end{bmatrix}\ngeq0$, $\Vert A^{-1}B\Vert=1.0910>1$, and $\rho(\vert A^{-1}B\vert)=0.9780<1$.

			\item[(2)] When $A=
			\begin{bmatrix}
				3 & 0\\
				0 & 3
			\end{bmatrix}$
			 and $B=
			\begin{bmatrix}
				-2 & 1\\
				1 & 2
			\end{bmatrix}$, we have $A^{-1}B=\dfrac{1}{3}
			\begin{bmatrix}
				-2 & 1\\
				1 & 2
			\end{bmatrix}\ngeq0$, $\Vert A^{-1}B\Vert=0.7454<1$, and $\rho(\vert A^{-1}B\vert)=1$.
		\end{enumerate}
\end{exam}
\end{rem}

If $M = \bar{M}+\Omega$, $N = \bar{N}+\Omega$, where $\Omega$ is a given matrix, $Q = Q_1 = I$ and $\tau = 1$, then the GNMS method changes into the NMS method and Corollary \ref{cor:nms} can be obtained.
\begin{cor}\label{cor:nms}
	If $\bar{M}+\Omega$ is nonsingular and
\begin{equation}\label{eq:nmsconv}
\Vert{(\bar{M}+\Omega)}^{-1}(\bar{N}+\Omega)\Vert+\Vert{(\bar{M}+\Omega)}^{-1}B\Vert <1
\end{equation}
then NMS method converges linearly to the unique solution $x^*$ of the GAVEs~\eqref{eq:gave}.
\end{cor}

\begin{proof}
 Under this circumstances, we can obtain \eqref{eq:nmsconv} from \eqref{eq:con2}.
 Hence, under \eqref{eq:nmsconv} we can conclude that the system \eqref{eq:gave} is unique solvable and the NMS method converges to the unique solution $x^*$.
\end{proof}
\begin{rem}\label{rem:nms}
In \cite[Theorem 4.1]{zhwl2021},  Zhou et al. show that the NMS iteration method (Algorithm \ref{alg:nms}) converges to a solution of the GAVEs~\eqref{eq:gave} if
\begin{equation}\label{eq:zw}
\Vert{(\bar{M}+\Omega)}^{-1}\Vert(\Vert\bar{N}+\Omega\Vert+\Vert B\Vert) <1
\end{equation}
As it can be seen \eqref{eq:zw} implies \eqref{eq:nmsconv} but the converse is generally not true. Thus the condition \eqref{eq:nmsconv} is weaker than \eqref{eq:zw} and we can conclude that the GAVEs \eqref{eq:gave} is uniquely solvable under the condition \eqref{eq:zw}.

\end{rem}

If $M = \theta\hat{M}+\hat{\Omega}$, $N =\hat{\Omega}+(\theta-1)\hat{M}+\hat{N}$, where $\hat{\Omega}$ is given matrix, $Q = Q_1 = I$ and $\tau = 1$, then the GNMS method reduces to the RNMS method and the following corollary can be obtained.
\begin{cor}\label{cor:rnms}
 If $\theta\hat{M}+\hat{\Omega}$ is nonsingular and
\begin{equation}\label{eq:rnmsconv}
\Vert{(\theta\hat{M}+\hat{\Omega})}^{-1}(\hat{\Omega}+(\theta-1)\hat{M}+\hat{N})\Vert+\Vert{(\theta\hat{M}+\hat{\Omega})}^{-1}B\Vert <1,
\end{equation}
then RNMS method (Algorithm \ref{alg:rnms}) converges linearly to the unique solution $x^*$ of the GAVEs~\eqref{eq:gave}.
\end{cor}

\begin{rem}
In \cite[Theorem 3.1]{zhsh2023}, Zhao and Shao present that the RNMS iteration method (Algorithm \ref{alg:rnms}) converges to a solution of the GAVEs~\eqref{eq:gave} if
\begin{equation}\label{eq:ws}
\Vert{(\theta\hat{M}+\hat{\Omega})}^{-1}\Vert(\Vert\hat{\Omega}+(\theta-1)\hat{M}+\hat{N}\Vert+\Vert B\Vert) <1
\end{equation}
It is easy to see that \eqref{eq:ws} implies \eqref{eq:rnmsconv}, but \eqref{eq:rnmsconv} generally not implies \eqref{eq:ws}. Thus \eqref{eq:rnmsconv} is weaker than \eqref{eq:ws}. By the Corollary \ref{cor:conv}, we can conclude the unique solvability of the GAVEs \eqref{eq:gave} when \eqref{eq:ws} holds.

\end{rem}

\section{ Numerical example}\label{sec:Numericalexample}
In this section, we utilize an example to demonstrate the effectiveness of the proposed method for solving the GAVEs \eqref{eq:gave}. All experiments were run on a personal computer with $3.20$ GHZ central processing unit \big(Intel (R), Corel(TM), i$5$-$11320$H\big), $16$GB memory and windows $11$ operating system, and the MATLAB version R$2021$a is used. Eight algorithms will be tested.
\begin{enumerate}

  \item GNMS: the Algorithm~\ref{alg:gnms} with $Q_1 = 10I$, $Q_2 = 0.5I$, $M = D-\frac{3}{4}L$ and $N = \frac{1}{4}L + U$, where $D$, $-L$ and $-U$ are the diagonal part,
the strictly lower-triangular and the strictly upper-triangular parts of $A$, respectively.

  \item MN: the Algorithm~\ref{alg:mn}.

  \item Picard: the Picard iteration \cite{rohf2014}
  $$
  x^{k+1} = A^{-1} (B|x^k| + c).
  $$

  \item FPI: the fixed point iteration
  \begin{equation*}
		\begin{cases}
		x^{k+1}=A^{-1}\left( By^k + c\right),\\
			y^{k+1}=(1-\tau)y^k + \tau |x^{k+1}|,
		\end{cases}
	\end{equation*}
which arises from \cite{ke2020}.

  \item NMS: the Algorithm~\ref{alg:nms} with $\bar{M} = M$.

  \item NGS: the Algorithm~\ref{alg:nms} with $\bar{M} = D - L$ and $\bar{N} = U$.

  \item RMS: the relaxed-based matrix splitting iteration \cite{soso2023}
      \begin{equation*}
		\begin{cases}
		y^{k+1}=S^{-1}\left(T x^k+ By^k + c\right),\\
			y^{k+1}=(1-\tau)y^k + \tau |x^{k+1}|
		\end{cases}
	\end{equation*}
with $S=M$ and $T = N$.

  \item SSMN: the Algorithm~\ref{alg:ssmn}.
\end{enumerate}

\begin{exam}\label{exam1}
Consider the GAVEs \eqref{eq:gave} with $A=\tilde{A}+\dfrac{1}{5}I$ and
$$B=
\left(
  \begin{array}{cccccccccccc}
    S_2 & -I & -I &-I &-I & 0 & 0 &0 & 0 & 0 & 0 & 0 \\
    -I & S_2 & -I &-I &-I & -I & 0 &0 & 0 & 0 & 0 & 0\\
    -I & -I & S_2 &-I &-I & -I & -I &0 & 0 & 0 & 0 & 0\\
    -I & -I & -I &S_2 &-I & -I & -I &-I & 0 & 0 & 0 & 0\\
    -I & -I & -I &-I &S_2 & -I & -I &-I & -I & 0 & 0 & 0\\
    0 & -I & -I &-I &-I & S_2 & -I &-I & -I & -I & 0 & 0\\
    \ddots&\ddots&\ddots&\ddots&\ddots&\ddots&\ddots
    &\ddots&\ddots&\ddots&\ddots&\ddots\\
    0 & 0&0&-I & -I &-I &-I & S_2 & -I &-I & -I & -I \\
    0 & 0&0&0&-I & -I &-I &-I & S_2 & -I &-I & -I  \\
    0 & 0&0&0&0&-I & -I &-I &-I & S_2 & -I &-I   \\
    0 & 0&0&0&0&0&-I & -I &-I &-I & S_2 & -I   \\
    0 & 0&0&0&0&0&0&-I & -I &-I &-I & S_2
  \end{array}
\right)
\in\mathbb{R}^{n\times n},$$
where	
$$\tilde{A}=\small
\left(
  \begin{array}{cccccccccccc}
    S_1 & -1.5I & -0.5I &-1.5I &-0.5I & 0 & 0 &0 & 0 & 0 & 0 & 0 \\
    -1.5I & S_1 & -1.5I &-0.5I &-1.5I & -0.5I & 0 &0 & 0 & 0 & 0 & 0\\
    -0.5I & -1.5I & S_1 &-1.5I &-0.5I & -1.5I & -0.5I &0 & 0 & 0 & 0 & 0\\
    -1.5I & -0.5I & -1.5I &S_1 &-1.5I & -0.5I & -1.5I &-0.5I & 0 & 0 & 0 & 0\\
    -0.5I & -1.5I & -0.5I &-1.5I &S_1 & -1.5I & -0.5I &-1.5I & -0.5I & 0 & 0 & 0\\
    0 & -0.5I & -1.5I &-0.5I &-1.5I & S_1 & -1.5I &-0.5I & -1.5I & -0.5I & 0 & 0\\
    \ddots&\ddots&\ddots&\ddots&\ddots&\ddots&\ddots
    &\ddots&\ddots&\ddots&\ddots&\ddots\\
    0 & 0&0&-0.5I & -1.5I &-0.5I &-1.5I & S_1 & -1.5I &-0.5I & -1.5I & -0.5I \\
    0 & 0&0&0&-0.5I & -1.5I &-0.5I &-1.5I & S_1 & -1.5I &-0.5I & -1.5I  \\
    0 & 0&0&0&0&-0.5I & -1.5I &-0.5I &-1.5I & S_1 & -1.5I &-0.5I   \\
    0 & 0&0&0&0&0&-0.5I & -1.5I &-0.5I &-1.5I & S_1 & -1.5I   \\
    0 & 0&0&0&0&0&0&-0.5I & -1.5I &-0.5I &-1.5I & S_1
  \end{array}
\right),
$$

$$
S_1=
\left(
  \begin{array}{cccccccccc}
    36 & -1.5 & -0.5 &-1.5 & 0 &0 & 0 & 0 & 0 & 0 \\
    -1.5 & 36 & -1.5 &-0.5 &-1.5 &  0 &0 & 0 & 0 & 0 \\
    -0.5 & -1.5 & 36 &-1.5 &-0.5 & -1.5 & 0 & 0 & 0 & 0\\
    -1.5 & -0.5 & -1.5 &36 &-1.5 & -0.5 & -1.5 & 0 & 0 & 0\\
    0 & -1.5 & -0.5 &-1.5 &36 & -1.5 & -0.5 &-1.5 & 0 & 0\\
    0 & 0 & -1.5 &-0.5 &-1.5 & 36 & -1.5 &-0.5 & -1.5  & 0\\
    \ddots&\ddots&\ddots&\ddots&\ddots&\ddots&\ddots
    &\ddots&\ddots&\ddots\\
     0&0&0&0 & -1.5 &-0.5 &-1.5 & 36 & -1.5 &-0.5   \\
     0&0&0&0&0& -1.5 &-0.5 &-1.5 &36 & -1.5   \\
     0&0&0&0&0&0& -1.5 &-0.5 &-1.5 & 36
  \end{array}
\right)
\in\mathbb{R}^{m\times m},$$

$$S_2=
\left(
  \begin{array}{cccccccccc}
    3 & -1 & -1 &-1 & 0 &0 & 0 & 0 & 0 & 0 \\
    -1 & 3 & -1 &-1 &-1 &  0 &0 & 0 & 0 & 0 \\
    -1 & -1 & 3 &-1 &-1 & -1 & 0 & 0 & 0 & 0\\
    -1 & -1 & -1 &3 &-1 & -1 & -1 & 0 & 0 & 0\\
    0 & -1 & -1 &-1 &3 & -1 & -1 &-1 & 0 & 0\\
    0 & 0 & -1 &-1 &-1 & 3 & -1 &-1 & -1  & 0\\
    \ddots&\ddots&\ddots&\ddots&\ddots&\ddots&\ddots
    &\ddots&\ddots&\ddots\\
     0&0&0&0 & -1 &-1 &-1 & 3 & -1 &-1   \\
     0&0&0&0&0& -1 &-1 &-1 &3 & -1   \\
     0&0&0&0&0&0& -1 &-1 &-1 & 3
  \end{array}
\right)
\in\mathbb{R}^{m\times m},$$

and $c=Ax^*-B|x^*|$ with $x^*=\left(\dfrac{1}{2},1,\dfrac{1}{2},1,\dots,\dfrac{1}{2},
1\right)^\top$.

In this example, we choose $x_0=(-1,0,-1,0,\dots,~-1,0)^\top$, $y_0= c$. For every   $n$, we run each method ten times and the average IT (the number of iteration), the average CPU (the elapsed CPU time in seconds) and the average RES are reported, where
$$
{\rm RES} = \frac{\|Ax^k - B|x^k| -c\|}{\|c\|}.
$$
Once ${\rm RES} \le 10^{-8}$, the experiment is terminated. Numerical results are shown in Table~\ref{table1}\footnote{In the table, $\tau_{opt}$ is the numerical optimal iteration parameter, which is selected from $[0:0.01:2]$ and is the first one to reach the minimal number of iteration of the method.}, from which we can see that the proposed method is the best one within the tested methods, both in terms of IT and CPU.

\setlength{\tabcolsep}{7.0pt}
\begin{table}[!h]\footnotesize \label{table1}
\centering
\caption{Numerical results for Example~\ref{exam1}.}\label{table1}
\begin{tabular}{ccccccc}\hline
			Method  & $m$ & $60$  & $80$ & $90$ & $100$ & $110$  \\ \hline
			GNMS   & $\tau_{opt}$  & $  1.00$ &  $ 1.00$ & $ 1.00$  & $ 1.00$  &$1.00$   \\
			&  IT  &  $8$ & $8$   & $8$    & $8$  &$8$   \\
			&   CPU & $0.0018$ &  $0.0033$ &  $0.0044$ & $0.0071$ &$0.0087$  \\
			&    RES &  4.1370e-09 & 3.1608e-09 & 2.8363e-09 & 2.5773e-09 & 2.3658e-09   \\
			MN    & $\Omega = 2*diag(A)$  &  &   &   & &     \\
			&  IT  &  $47$ & $47$   & $47$    & $47$  &$47$   \\
			&   CPU & $0.3717$ &  $0.7608$ &  $1.0522$ & $1.4283$ &$1.7924$  \\
			&    RES &  7.5124e-09 & 7.0945e-09 & 6.9526e-09 & 6.8380e-09 &  6.7435e-09 \\
            & $\Omega = \frac{1}{2}*diag(A)$  &  &   &   & &     \\
			&  IT  &  $16$ & $16$   & $16$    & $16$  &$16$   \\
			&   CPU & $0.1268$ &  $0.2544$ &  $0.3548$ & $0.4895$ &$0.6081$  \\
			&    RES &  7.2195e-09 & 5.9416e-09 & 5.5203e-09 & 5.1856e-09 &  4.9137e-09 \\
			Picard   &   &  &   &   &   &   \\
			&  IT  &  $26$ & $26$   & $26$    & $26$ & $26$   \\
			&   CPU & $0.2003$ &  $0.4274$ &  $0.5773$ & $0.7880$ &$0.9853$  \\
			&    RES &  6.9693e-09 & 8.2848e-09 & 8.7217e-09 & 9.0704e-09 & 9.3553e-09   \\
           FPI   &  $\tau_{opt}$ & $0.8$ & $0.8$  &  $0.79$ &  $0.79$ & $0.79$  \\
			&  IT  &  $17$ & $17$   & $17$    & $17$ & $17$   \\
			&   CPU & $0.1329$ &  $0.2729$ &  $0.3863$ & $0.5147$ &$0.6506$  \\
			&    RES &  9.2742e-09 & 8.4833e-09 & 9.7848e-09 & 9.3634e-09 & 8.9942e-09   \\
			NMS    & $\Omega = 2*diag(A)$  &  &   &   &   &   \\
			&  IT  &  $52$ & $52$   & $52$    &  $52$ & $52$   \\
			&   CPU & $0.0079$ &  $0.0164$ &  $0.0252$ & $0.0416$ & $0.0493$  \\
			&    RES &  7.8099e-09 & 7.7233e-09 & 7.6941e-09 & 7.6705e-09 & 7.6512e-09   \\
            & $\Omega = \dfrac{1}{2}*diag(A)$  &  &   &   &   &   \\
			&  IT  &  $19$ & $19$   & $19$    &  $19$ & $19$   \\
			&   CPU & $0.0032$ &  $0.0066$ &  $0.0090$ & $0.0162$ & $0.0193$  \\
			&    RES &  5.6173e-09 & 5.1661e-09 & 5.0093e-09 & 4.8812e-09 & 4.7744e-09   \\
             NGS    & $\Omega = 2*diag(A)$  &  &   &   &   &   \\
			&  IT  &  $51$ & $51$   & $51$    &  $51$ &$51$   \\
			&  CPU & $0.0071$ &  $0.0142$ &  $0.0202$ & $0.0324$ &$0.0423$  \\
			&  RES &  7.6531e-09 & 7.5154e-09 & 7.4677e-09 & 7.4300e-09 & 7.3991e-09   \\
             & $\Omega = \dfrac{1}{2}*diag(A)$  &  &   &   &   &   \\
			&  IT  &  $18$ & $18$   & $18$  &  $18$ &$18$   \\
			&  CPU & $0.0029$ &  $0.0057$ &  $0.0074$ & $0.0137$ &$0.0160$  \\
			&  RES &  8.0587e-09 & 7.0044e-09 & 6.6319e-09 & 6.3241e-09 & 6.0648e-09   \\
			RMS  &$\tau_{opt}$ & $ 0.99$ & $0.99$ & $ 0.99$ & $ 0.99$ &$0.99$   \\
			&  IT  &  $12$ & $12$   & $12$    & $12$  &$12$   \\
			&  CPU & $0.0022$ &  $0.0046$ &  $0.0061$ & $0.0093$ &$0.0115$  \\
			&  RES &  3.4193e-09 & 2.7439e-09 & 2.5157e-09 & 2.3315e-09 & 2.1795e-09   \\
       SSMN & $\tilde{\Omega} = 2*diag(A)$ &  &  &  &  & \\
			&  IT  &  $18$ & $18$   & $18$  &  $18$ & $18$   \\
			&  CPU & $0.1443$ &  $0.2917$ &  $0.4124$ & $0.5657$ &$0.7034$  \\
			&  RES &  5.0798e-09 & 4.5585e-09 & 4.3772e-09 &  4.2288e-09 & 4.1049e-09   \\
            & $\tilde{\Omega} = \dfrac{1}{2}*diag(A)$ &  &  &  &  & \\
			&  IT  &  $39$ & $39$   & $39$  &  $39$ & $39$   \\
			&  CPU & $0.3126$ &  $0.6474$ &  $0.8856$ & $1.2022$ &$1.5239$   \\
			&  RES &  7.7547e-09  & 8.7984e-09  & 9.1439e-09  & 9.4195e-09  & 9.6445e-09  \\\hline
\end{tabular}
\end{table}

\end{exam}

\section{Conclusions}\label{sec:Conclusions}
A generalization of the Newton-based matrix splitting iteration method (GNMS) for solving GAVEs is proposed, which include some existing methods as special cases. Under mild conditions, the GNMS method converges
to the unique solution of the GAVEs and  a few weaker
convergence conditions for some existing methods are obtained. Numerical results illustrate that GNMS can be superior to some existing methods in our setting.

\bibliography{cjcmsample}

\begin{thebibliography}{99}

\bibitem{acha2018}
M. Achache, N. Hazzam. Solving absolute value equations via complementarity and interior-point methods, \emph{J. Nonlinear Funct. Anal.}, 2018: 1--10, 2018.

\bibitem{alct2023}
J. H. Alcantara, J.-S. Chen, M. K. Tam. Method of alternating projections for the
general absolute value equation, \emph{J. Fixed Point Theory Appl.}, 25: 39, 2023.

\bibitem{cyhm2023}
C.-R. Chen, D.-M. Yu, D.-R Han. Exact and inexact Douglas-Rachford splitting methods for solving large-scale
sparse absolute value equations, \emph{IMA J. Numer. Anal.}, 43: 1036--1060, 2023.

\bibitem{copa1992}
R. W. Cottle, J.-S. Pang. R. E. Stone. The Linear Complementarity Problem, \emph{Academic Press, New York.}, 1992.

\bibitem{hlad2018}
M. Hlad\'{i}k. Bounds for the solutions of absolute value equations, \emph{Comput. Optim. Appl.}, 69: 243--266, 2018.

\bibitem{huhu2010}
S.-L. Hu, Z.-H. Huang. A note on absolute value equations, \emph{Optim. Lett.}, 4: 417--424, 2010.
	


\bibitem{jizh2013}
X.-Q. Jiang, Y. Zhang. A smoothing-type algorithm for absolute value equations, \emph{J. Ind. Manag. Optim.}, 9: 789--798, 2013.
	
\bibitem{ke2020}
Y.-F. Ke. The new iteration algorithm for absolute value equation, \emph{Appl. Math. Lett.}, 99: 105-990, 2020.


\bibitem{lilw2018}
Y.-Y. Lian, C.-X. Li, S.-L. Wu. Weaker convergent results of the generalized Newton method for the generalized absolute value equations, \emph{J. Comput. Appl. Math.}, 338: 221--226, 2018.

\bibitem{liyi2021}
X. Li, X.-X. Yin. A new modified Newton-type iteration method for solving generalized absolute value equations, \emph{arXiv:2103.09452v3}, 2021. \href{
https://doi.org/10.48550/arXiv.2103.09452}{
https://doi.org/10.48550/arXiv.2103.09452}.

\bibitem{lild2022}
X. Li, Y.-X. Li, Y. Dou. Shift-splitting fixed point iteration method for solving generalized absolute value equations, \emph{Numer. Algor.}, 2022.
\href{https://doi.org/10.1007/s11075-022-01435-3}{https://doi.org/10.1007/s11075-022-01435-3}.

\bibitem{harj1985}
Roger A. Horn, Charles R. Johnson. Matrix analysis (Second edition). \emph{Cambridge University Press, New York.},1985.

\bibitem{love2013}
T. Lotfi, H. Veiseh. A note on unique solvability of the absolute value equation, \emph{J. Linear Topological Algebra}, 2: 77--81, 2013.
	
\bibitem{mame2006}
O. L. Mangasarian, R. R. Meyer. Absolute value equations, \emph{Linear Algebra Appl.}, 419: 359--367, 2006.

\bibitem{mang2007}
O. L. Mangasarian. Absolute value programming, \emph{Comput. Optim. Appl.}, 36: 43--53, 2007.
	
\bibitem{mezz2020}
F. Mezzadri. On the solution of general absolute value equations, \emph{Appl. Math. Lett.}, 107: 106462, 2020.

\bibitem{prok2009}
O. Prokopyev. On equivalent reformulations for absolute value equations, \emph{Comput. Optim. Appl.}, 44: 363--372, 2009.
	
\bibitem{rohn2004}
J. Rohn. A theorem of the alternatives for the equation $Ax + B|x| = b$, \emph{Linear  Multilinear Algebra}, 52: 421--426, 2004.	
\bibitem{rohn2009}
J. Rohn. On unique solvability of the absolute value equation, \emph{Optim. Lett.}, 3: 603--606, 2009.

\bibitem{rohf2014}
J. Rohn, V. Hooshyarbakhsh, R. Farhadsefat. An iterative method for solving absolute value equations and sufficient conditions for unique solvability, Optim. Lett., 8: 35--44, 2014.	

	
\bibitem{serre2002}	
D. Serre. Matrices: Theory and Applications, \emph{Springer, New York.}, 2002.
	
\bibitem{shzh2023}
X.-H. Shao, W.-C. Zhao. Relaxed modified Newton-based iteration method for generalized absolute value equations, \emph{AIMS Math.}, 8: 4714--4725, 2023.

\bibitem{soso2023}
J. Song, Y.-Z. Song. Relaxed-based matrix splitting methods for solving absolute value equations, \emph{Comp. Appl. Math.}, 42: 19, 2023.

\bibitem{tazh2019}
J.-Y. Tang, J.-C. Zhou. A quadratically convergent descent method for the absolute value equation
$Ax + B|x| = b$, \emph{Oper. Res. Lett.}, 47: 229--234, 2019.

\bibitem{wacc2019}
A. Wang, Y. Cao, J.-X. Chen. Modified Newton-type iteration methods for generalized absolute value equations, \emph{J. Optim. Theory. Appl.}, 181: 216--230, 2019.

\bibitem{wuli2020}
S.-L. Wu, C.-X. Li. A note on unique solvability of the absolute value equation, \emph{Optim. Lett.}, 14: 1957--1960, 2020.

\bibitem{wush2021}
S.-L. Wu, S.-Q. Shen. On the unique solution of the generalized absolute value equation, \emph{Optim. Lett.}, 15: 2017--2024, 2021.

\bibitem{ys2003}
Y. Saad. Iterative Methods for Sparse Linear Systems (Second edition), \emph{SIAM}, USA, 2003.

\bibitem{young1971}
D.-M. Young. Iterative Solution of Large Linear Systems, \emph{Academic Press, New York}, 1971.
	
\bibitem{yuch2022}	
D.-M. Yu, C.-R. Chen, D.-R. Han. A modified fixed point iteration method for solving the system of absolute value equations, \emph{Optimization}, 71: 449-461, 2022.


\bibitem{zhwl2021}	
H.-Y. Zhou, S.-L. Wu, C.-X. Li. Newton-based matrix splitting method for generalized absolute value equation, \emph{J. Comput. Appl. Math.}, 394: 113578, 2021.

\bibitem{zhzl2023}
J.-L. Zhang, G.-F. Zhang, Z.-Z. Liang. A modified generalized SOR-like method for solving an absolute value equation, \emph{Linear Multilinear Algebra}, 2022. \href{DOI: 10.1080/03081087.2022.2066614}{DOI: 10.1080/03081087.2022.2066614}.

\bibitem{zhsh2023}
W.-C. Zhao, X.-H. Shao. New matrix splitting iteration method for generalized absolute value equations, \emph{AIMS Math.}, 8(5): 10558-10578, 2023.


\end{thebibliography}

\end{document}